\documentclass[a4paper,twoside,10pt]{amsart}

\usepackage{amssymb}

\ifx\isdraft\undefined
\else
    \usepackage{fancyhdr}
\fi

\numberwithin{equation}{section}

\newtheorem{theorem}{Theorem}[section]

\newtheorem{claim}[theorem]{Claim}

\newtheorem{lemma}[theorem]{Lemma}

\newtheorem{question}[theorem]{Question}

\theoremstyle{definition}

\newtheorem{remark}[theorem]{Remark}

\newcommand{\<}{\langle}
\newcommand{\uh}{\upharpoonright}
\renewcommand{\>}{\rangle}

\newcommand{\dom}{\operatorname{dom}}

\newcommand{\RCA}{\operatorname{RCA}_0}

\newcommand{\RT}{\operatorname{RT}}
\newcommand{\COH}{\operatorname{COH}}

\newcommand{\SRT}{\operatorname{SRT}}

\begin{document}

\title{Omitting Cohesive Sets}



\author{Wei Wang}

\thanks{This research is partially supported by NSF Grant 11001281 of China and an NCET grant from the Ministry of Education of China. The author thanks Mingzhong Cai for bringing an interesting question to his attention.}

\address{Institute of Logic and Cognition and Department of Philosophy, Sun Yat-sen University, 135 Xingang Xi Road, Guangzhou 510275, P.R. China}
\email{wwang.cn@gmail.com}

\begin{abstract}
We prove that if $\vec{R}$ is a computable sequence of subsets of $\omega$ which admits no computable cohesive set, then no $3$-generic computes any $\vec{R}$-cohesive set; and there exists a Martin-L\"{o}f random which computes no $\vec{R}$-cohesive set.
\end{abstract}

\ifx\isdraft\undefined
\else
    \today
\fi

\maketitle

\section{Introduction}

For two sets $X$ and $Y$, we write $X \subseteq^* Y$ if $X - Y$ is finite, and write $X =^* Y$ if $X \subseteq^* Y$ and $Y \subseteq^* X$. For $\vec{R} = (R_n: n \in \omega)$ an infinite sequence of subsets of $\omega$, an infinite set $X \subseteq \omega$ is \emph{$\vec{R}$-cohesive} if for all $n$ either $X \subseteq^* R_n$ or $X \subseteq^* \omega - R_n$. In reverse mathematics, the existence of cohesive sets ($\COH$) is an easy consequence of Ramsey Theorem for Pairs ($\RT^2_2$) over $\RCA$. $\COH$ was introduced by Cholak, Jockusch and Slaman \cite{Cholak.Jockusch.ea:2001.Ramsey} and turns out to be a useful principle with rich properties (e.g., see \cite{Cholak.Jockusch.ea:2001.Ramsey} and \cite{Hirschfeldt.Shore:2007}). Cholak, Jockusch and Slaman prove that $\COH$ is strictly weaker than $\RT^2_2$ over $\RCA$ and raise the question whether $\RCA + \SRT^2_2 \vdash \COH$, where $\SRT^2_2$ is the stable version of $\RT^2_2$. This question remains a major open question of the subject.

There have been some attempts to understand the complexity of cohesive sets. Jockusch and Stephan \cite{Jockusch.Stephan:1993.cohesive} construct a primitively computable $\vec{R}$ such that $C' \gg \emptyset'$ for every $\vec{R}$-cohesive $C$. Recall that $X \gg Y$ if $X$ computes a function $f: \omega \to 2$ such that $f(e) \neq \Phi_e(Y;e)$ for all $e$. In the same paper, Jockusch and Stephan also prove that if $X' \gg \emptyset'$ then every computable sequence admits an $X$-computable cohesive set $C$.

From another viewpoint, recently Mingzhong Cai asks the following.

\begin{question}[Cai] \label{qst:Cai}
Suppose that $\vec{R}$ is computable but admits no computable cohesive set. Does there exist a non-computable $X$ which computes \emph{no} $\vec{R}$-cohesive set?
\end{question}

Cai's question can be put in the following way: does there exist a computable $\vec{R}$ such that the Turing degrees of $\vec{R}$-cohesive sets are exactly the non-computable ones?

In this paper, we answer Question \ref{qst:Cai} affirmatively. We show that if $\vec{R}$ is computable and admits no computable cohesive set, then no $3$-generic computes any $\vec{R}$-cohesive set (\S2), and there exists a Martin-L\"{o}f random which computes no $\vec{R}$-cohesive set (\S3).

We close this section by introducing some notation.

We use lower case Greek letters for finite binary strings. For $b_0, b_1, \ldots, b_{n-1} < 2$, we write $\<b_0 b_1 \ldots b_{n-1}\>$ for the string $\sigma$ of length $n$ such that $\sigma(i) = b_i$ for all $i < n$. If $\sigma, \tau \in 2^{<\omega}$, we write $\sigma\tau$ for the concatenation of $\sigma$ and $\tau$, i.e., the string $\eta$ such that $|\eta| = |\sigma| + |\tau|$, $\eta(i) = \sigma(i)$ for $i < |\sigma|$ and $\eta(|\sigma| + i) = \tau(i)$ for $i < |\tau|$. We write $\sigma \prec \tau$ or $\tau \succ \sigma$ if $\sigma$ is a proper initial segment of $\tau$, and $\sigma \preccurlyeq \tau$ or $\tau \succcurlyeq \sigma$ if either $\sigma \prec \tau$ or $\sigma = \tau$. A subset $D \subseteq 2^{<\omega}$ is \emph{dense} if every $\sigma$ is an initial segment of some $\tau \in D$; $D$ is \emph{dense below $\sigma$} if $D \cup \{\tau: \sigma \not\prec \tau\}$ is dense. A binary \emph{tree} $T$ is a subset of $2^{<\omega}$ such that $\tau \prec \sigma \in T \to \tau \in T$. If $T$ is a tree and $\
sigma \in T$ then let $T(\sigma) = \{\tau: \sigma
\prec \tau \text{ or } \tau \preccurlyeq \sigma\}$.

For a sequence $\vec{R} = (R_n: n \in \omega)$ and each $\nu \in 2^{<\omega}$, let
$$
    R_\nu = \bigcap_{\nu(i) = 1} R_i \cap \bigcap_{\nu(i) = 0} (\omega - R_i).
$$

Readers may refer to \cite{Soare:1987.book}, \cite{Downey.Hirschfeldt:2010.book} and \cite{Simpson:1999.SOSOA} for more computability, algorithmic randomness and reverse mathematics background. For reverse mathematics of Ramsey theory, \cite{Cholak.Jockusch.ea:2001.Ramsey} and \cite{Hirschfeldt.Shore:2007} are good sources.

\section{Cohesive Sets and Cohen Generics}

\begin{theorem}\label{thm:coh-generic}
If $\vec{R} = (R_n: n \in \omega)$ is a computable sequence admitting no computable cohesive set, then \emph{no} $3$-generic computes $\vec{R}$-cohesive sets.
\end{theorem}

We prove the above theorem by establishing some density lemmata.

Fix a functional $\Phi_e$. Let
$$
    D_{e,0} = \{\tau: \exists x \forall \rho \succ \tau(|\dom \Phi_e(\rho)| < x)\},
$$
and let
$$
    D_{e,1,n} = \{\tau: \forall \rho \succ \tau, x \exists \zeta \succ \rho, u,v > x(u,v \in \dom \Phi_e(\zeta) \wedge R_n(u) \neq R_n(v))\}.
$$

\begin{lemma}\label{lem:coh-generic-density}
For every $e$, $D_{e,0} \cup \bigcup_n D_{e,1,n}$ is dense.
\end{lemma}

\begin{proof}
For a contradiction, fix $\sigma$ such that $\forall \tau \succ \sigma(\tau \not\in D_{e,0} \cup \bigcup_n D_{e,1,n})$. We build a computable $\vec{R}$-cohesive $C$ by a finite injury argument.

At stage $0$, let $l_0 = 0$, $\sigma_{0,0} = \sigma$ and $\nu_0 = C_0 = \emptyset$.

At stage $s+1$, suppose that we have defined $l_s$, $\nu_s \in 2^{l_s}$, $(\sigma_{i,s}: i \leq l_s)$, $(x_{i,s}: i < l_s)$ and $C_s$ such that $\sigma_{i,s} \preccurlyeq \sigma_{i+1,s}$ for $i < l_s$ and $C_s$ is finite. Wait until one of the following statements holds:
\begin{enumerate}
    \item for some $\zeta \succ \sigma_{l_s,s}$ and $u > \max C_s$, $u \in \dom \Phi_e(\zeta) \cap R_{\nu_s}$;
    \item for some $i < l_s$, $\zeta \succ \sigma_{i+1,s}$ and $u > x_{i,s}$, $u \in \dom \Phi_e(\zeta) - R_{\nu_s \uh (i+1)}$.
\end{enumerate}

If (1) holds, then let $C_{s+1} = C_s \cup \{u\}$. Let $l_{s+1} = l_s + 1$, $\nu_{s+1} = \nu_s \<0\>$, $(\sigma_{i,s+1}: i \leq l_s) = (\sigma_{i,s}: i \leq l_s)$ and $\sigma_{l_s+1,s+1} = \sigma_{l_s,s}$, and $(x_{i,s+1}: i < l_s) = (x_{i,s}: i < l_s)$ and $x_{l_s,s+1} = x_{l_s-1,s}$. Goto stage $s+2$.

Suppose that (2) holds, fix the least $i$ as in (2) that has been observed. Let $(b, \sigma_{i+1,s+1}, x_{i,s+1})$ be the least triple greater than $(\nu_s(i), \sigma_{i+1,s}, x_{i,s})$ such that $b < 2$, $\sigma_{i+1,s+1} \succcurlyeq \sigma_{i,s}$ and $x_{i,s+1} \geq x_{i,s}$. Let $l_{s+1} = i+1$, $\nu_{s+1} = (\nu_s \uh i) \<b\>$, $(\sigma_{j,s+1}: j \leq i) = (\sigma_{j,s}: j \leq i)$ and $(x_{j,s+1}: j < i) = (x_{j,s}: j < i)$. Let $C_{s+1} = C_s$ and goto stage $s+2$.

Finally, let $C = \bigcup_s C_s$.

Intuitively, we guess that $(\nu_s(i), \sigma_{i+1,s}, x_{i,s})$ is a witness for $\sigma_{i,s}$ being not in $D_{e,1,i}$. In other words, we guess that
$$
    \forall \zeta \succ \sigma_{i+1,s}, u > x_{i,s} (\Phi_e(\zeta; u) \downarrow \to R_i(u) = \nu_s(i)).
$$
If our speculation turns out to be wrong, then we make a new guess. As $\sigma_{i,s} \not\in D_{e,1,i}$ and the set of witnesses is $\Pi^0_1$, eventually a correct one will fall into our hands, unless $\sigma_{i,s}$ is changed.

\begin{claim}
The construction is not blocked at any stage.
\end{claim}

\begin{proof}
If the following statement holds
$$
    \forall \zeta \succ \sigma_{l_s,s}, u > x_{l_s - 1, s} (\Phi_e(\zeta; u) \downarrow \to u \in R_{\nu_s}),
$$
then (1) holds for some $\zeta$ and $u$, as $\sigma_{l_s,s} \not\in D_{e,0}$. So the construction proceeds to stage $s+2$.

Otherwise, (2) holds for some $i, \zeta$ and $u$. So, either we find such a triple in stage $s+1$, or (1) holds before such a triple is discovered. In either case, the construction proceeds to stage $s+2$.
\end{proof}

\begin{claim}
For every $i$, all following limits exist
$$
    \nu(i) = \lim_s \nu_s (i), \sigma_{i} = \lim_s \sigma_{i,s}, x_i = \lim_s x_{i,s}.
$$
\end{claim}

\begin{proof}
Clearly, $\sigma_0 = \lim_s \sigma_{0,s} = \sigma$.

Suppose that all following limits exist
$$
    \sigma_j = \lim_s \sigma_{j,s} \text{ for } j \leq i, \nu(j) = \lim_s \nu_s(j) \text{ and } x_j = \lim_s x_{j,s} \text{ for } j < i.
$$
As $\sigma_i \not\in D_{e,1,n}$, there exist $\rho \succ \sigma_i$, $x > \max \{x_j: j < i\}$ and $b < 2$ such that
$$
    \forall \zeta \succ \rho, u > x(u \in \dom \Phi_e(\zeta) \to R_i(u) = b).
$$
So, the construction eventually will encounter such triple and the limits $\lim_s \nu_s(i)$, $\lim_s \sigma_{i+1,s}$ and  $\lim_s x_{i,s}$ exist.
\end{proof}

By the above claims, $l_s \to \infty$ when $s \to \infty$. As $C_s$ increases whenever $l_s$ increases, $C$ is infinite. It follows from the construction that $C$ is computable. As $\nu(i)$ exists for all $i$, $C$ is $\vec{R}$-cohesive.
\end{proof}

\begin{lemma}
If $\tau \in D_{e,1,n} - D_{e,0}$ then for every $x$ the following set is dense below $\tau$:
$$
    \{\zeta: \exists u,v > x(u,v \in \dom \Phi_e(\zeta) \wedge R_n(u) \neq R_n(v))\}.
$$
\end{lemma}

\begin{proof}
Immediately.
\end{proof}

Theorem \ref{thm:coh-generic} follows from the above lemmata.

\begin{remark}
Note that there are $1$-generics with jumps $\gg \emptyset'$, by Friedberg's Jump Inversion Theorem \cite[Theorem VI.3.1]{Soare:1987.book}. Hence a $1$-generic may compute a cohesive set for every computable sequence, by Jockusch and Stephan \cite{Jockusch.Stephan:1993.cohesive}. However, the situation for $2$-generics is unkown.
\end{remark}

\section{Cohesive Sets and Martin-L\"{o}f Randoms}

In this section, we prove the following theorem.

\begin{theorem}\label{thm:coh-random}
If $\vec{R} = (R_n: n \in \omega)$ is a computable sequence which admits no computable cohesive set, then there exists a Martin-L\"{o}f random $X$ which computes no $\vec{R}$-cohesive set.
\end{theorem}

Let $m$ denote the canonical Lebesgue measure on Cantor space. Fix a computable enumeration $(T_e: e \in \omega)$ of all computable binary trees. We need a coding technique which was first introduced by Ku\v{c}era \cite{Kucera:85}. The version that we need is due to Reimann (see \cite[\S8.5]{Downey.Hirschfeldt:2010.book}). \footnote{The author thanks Frank Stephan and Jing Zhang for pointing out a mistake in the statement of the following theorem.}

\begin{theorem}[Ku\v{c}era coding] \label{thm:Kucera-coding}
There exist a computable binary tree $U$ and a computable function $g: \omega \to \mathbf{Q} \cap (0,1)$ such that $m [U] > 0$ and for all $e$
$$
    [T_e] \cap [U] \neq \emptyset \leftrightarrow m ([T_e] \cap [U]) > g(e).
$$
\end{theorem}

Fix $U$ and $g$ as in Theorem \ref{thm:Kucera-coding}. The plan is to build a descending sequence $(S_n: n \in \omega)$ of computable subtrees of $U$ such that every $X \in \bigcap_n [S_n]$ is a desired random for Theorem \ref{thm:coh-random}.

Let $\mathbb{P}$ be the set of infinite computable subtrees of $U$. By Theorem \ref{thm:Kucera-coding}, if $T \in \mathbb{P}$ then $m[T] > 0$. We order trees in $\mathbb{P}$ by inclusion and define density accordingly. For each $e$ and $x$, let
$$
    E_{e,0,x} = \{T \in \mathbb{P}: \forall Y \in [T] (|\dom \Phi_e(Y)| < x)\},
$$
and for each $e$ and $n$, let
$$
    E_{e,1,n} = \{T \in \mathbb{P}: m\{Y \in [T]: \dom \Phi_e(Y) \subseteq^* R_n \text{ or } \dom \Phi_e(Y) \subseteq^* \omega - R_n\} = 0\}.
$$
As an immediate observation, if $T \in E_{e,0,x}$ ($T \in E_{e,1,n}$) and $T' \subseteq T$ in $\mathbb{P}$ then $T' \in E_{e,0,x}$ ($T' \in E_{e,1,n}$).

We establish a parallel of Lemma \ref{lem:coh-generic-density}.

\begin{lemma}\label{lem:coh-rand-density}
For each $e$, $\bigcup_x E_{e,0,x} \cup \bigcup_n E_{e,1,n}$ is dense.
\end{lemma}

\begin{proof}
For a contradiction, let $e$ and $S \in \mathbb{P}$ be such that there is no $T \subseteq S$ in $\bigcup_x E_{e,0,x} \cup \bigcup_n E_{e,1,n}$. It follows that
$$
    m \{Y \in [S]: |\dom \Phi_e(Y)| < \omega\} = 0.
$$
Moreover, for all $n$ and $T \subseteq S$ with $T \in \mathbb{P}$, there exists $i < 2$ such that
$$
    m \{Y \in [T]: \forall y > x (\Phi_e(Y; y) \downarrow \to R_n(y) = i)\} > 0
$$
for all sufficiently large $x$. We build a computable $\vec{R}$-cohesive set $C$ by a finite injury argument.

At stage $0$, let $i < 2$ and $x$ be such that $m[T] > 0$, where
$$
    T = \{\sigma \in S: \forall y > x (\Phi_e(\sigma; y) \downarrow \to R_0(y) = i)\}.
$$
Let $e_{0,0}$ and $e_{0,1}$ be such that $T_{e_{0,0}} = S$ and $T_{e_{0,1}} = T$, and let $k_0 = 1$, $\nu_0 = \<i\>$, $x_{0,0} = x$ and $C_0 = \emptyset$.

At stage $s+1$, suppose that we have defined $k_s > 0$, $\nu_s \in 2^{k_s}$, $(e_{i,s}: i \leq k_s)$, $(x_{i,s}: i < k_s)$ and $C_s$ such that $C_s$ is finite, $e_{0,s} = e_{0,0}$ and for all $i < k_s$
$$
    T_{e_{i+1,s}} = \{\sigma \in T_{e_{i,s}}: \forall y > x_{i,s}(\Phi_e(\sigma; y) \downarrow \to y \in  R_{\nu_s \uh (i+1)})\}.
$$
Wait until one of the following statements holds:
\begin{enumerate}
    \item $m [T_{e_{j,s}}] \leq g(e_{j,s})$ for some $j \leq k_s$;
    \item for some $y > \max C_s$ and $\sigma \in T_{e_{k_s,s}}$, $\Phi_e(\sigma; y) \downarrow$ and $y \in R_{\nu_s}$.
\end{enumerate}

If (1) holds, then fix the least $j$ that has been observed. As $m[S] > 0$, $j > 0$ by Theorem \ref{thm:Kucera-coding}. Let $(x_{j-1,s+1},b) \in \omega \times \{0,1\}$ be the least pair greater than $(x_{j-1,s},\nu_s(j-1))$ with $x_{j-1,s+1} \geq x_{j-1,s}$. Let $k_{s+1} = j$, $\nu_{s+1} = (\nu_s \uh (j-1)) \<b\>$, $e_{i,s+1} = e_{i,s}$ for $i < j$, $e_{j,s+1}$ be such that
$$
    T_{e_{j,s+1}} = \{\sigma \in T_{e_{j-1,s+1}}: \forall y > x_{j-1,s+1}(\Phi_e(\sigma; y) \downarrow \to y \in R_{\nu_{s+1}})\}.
$$
For $i < j-1$, let $x_{i,s+1} = x_{i,s}$. Let $C_{s+1} = C_s$ and goto stage $s+2$.

If (2) holds, then let $C_{s+1} = C_s \cup \{y\}$. Let $k_{s+1} = k_s + 1$, $\nu_{s+1} = \nu_s \<0\>$, $x_{i,s+1} = x_{i,s}$ for $i < k_s$ and $x_{k_s,s+1} = x_{k_s - 1,  s}$, $e_{i,s+1} = e_{i,s}$ for $i \leq k_s$ and $e_{k_s + 1, s+1}$ be such that
$$
    T_{e_{k_s + 1,s+1}} = \{\sigma \in T_{e_{k_s,s+1}}: \forall y > x_{k_s,s+1}(\Phi_e(\sigma; y) \downarrow \to y \in R_{\nu_{s+1}})\}.
$$
Goto stage $s+2$.

Finally, let $C = \bigcup_s C_s$.

We prove that the construction does produce a desired $C$.

\begin{claim}
The construction is \emph{not} blocked at any stage.
\end{claim}

\begin{proof}
At stage $s+1$, suppose that $m [T_{e_{k_s,s}}] = 0$. Then, either (1) or (2) holds at last and the construction proceeds to stage $s+2$.

Assume that $m [T_{e_{k_s,s}}] > 0$. Then $m [T_{e_{k_s,s}}] > g(e_{k_s,s})$ and
$$
    m \{Y \in [T_{e_{k_s,s}}]: |\dom \Phi_e(Y)| = \omega\} = m [T_{e_{k_s,s}}] > g(e_{k_s,s}).
$$
So, there exist $y$ sufficiently large and $\sigma \in T_{e_{k_s,s}}$ such that $\Phi_e(\sigma; y) \downarrow$. By the inductive hypothesis of the construction, $y \in R_{\nu_s}$. Hence, the construction proceeds to stage $s+2$.
\end{proof}

\begin{claim}
When $s \to \infty$, $k_s \to \infty$ and the limits $\lim_s e_{i,s}$ exist for all $i$.
\end{claim}

\begin{proof}
We prove the claim by induction.

Clearly, $k_s \geq 0$ and $e_{0,s} = e_{0,0}$ for all $s$. Below, fix $i$ and $s_0$ such that $k_s > i$ and $e_{j,s} = e_{j,s_0}$ for all $j \leq i$ and $s > s_0$. Let $e_i = \lim_s e_{i,s}$. It follows that
$$
    m \{Y \in [T_{e_i}]: |\dom \Phi_e(Y)| = \omega\} = m [T_{e_i}] > g(e_i) > 0.
$$
Clearly, $x_j = \lim_s x_{j,s}$ exists for each $j < i$. Let $\bar{x} = \max \{x_j: j < i\}$. If $k_s = i + 1$ for sufficiently many $s > s_0$, then there exist $s_1 > s_0$, $x \geq \bar{x}$ and $b < 2$ such that
$$
    [T_{e_{i+1,s_1}}] = \{Y \in [T_{e_i}]: \forall y > x(\Phi_e(Y;y) \downarrow \to R_i(y) = b)\}
$$
and $m [T_{e_{i+1,s_1}}] > g(e_{i+1,s_1})$. Hence, $e_{i+1,s} = e_{i+1,s_1}$ and $k_s > i + 1$ for all $s > s_1$.
\end{proof}

Obviously, $C$ is computable. We show that $C$ is $\vec{R}$-cohesive. By the above claim and the construction, $C$ is infinite. The above claim implies that $\lim_s \nu_s(i)$ exists for all $i$. Hence, $C \subseteq^* R_i$ or $C \subseteq^* \omega - R_i$ for all $i$.
\end{proof}

\begin{proof}[Proof of Theorem \ref{thm:coh-random}]
Let $S_0 = U$.

Suppose that we have $S_k \in \mathbb{P}$ and two partial sequences $(x_i: i < k)$ and $(n_i: i < k)$ such that
\begin{enumerate}
    \item for each $i < k$ exactly one of $x_i$ and $n_i$ is defined,
    \item if $x_i$ is defined then $S_k \in E_{i,0,x_i}$,
    \item if $n_i$ is defined then $S_k \in E_{i,1,n_i}$ and $\forall T \subseteq S_k(T \not\in \bigcup_x E_{i,0,x})$.
\end{enumerate}
Let $\sigma \in S_k$ be such that $m [S_k(\sigma)] > 0$ and
$$
    \exists u_i, v_i > k (u_i, v_i \in \dom \Phi_i(\sigma) \wedge R_{n_i}(u_i) \neq R_{n_i}(v_i))
$$
for all $i < k$ with $n_i$ defined. The existence of $\sigma$ is guaranteed by (3) above.

Let $S_{k+1} \subseteq S_k(\sigma)$ be such that $S_{k+1} \in E_{k,0,x_{k}}$ for some $x_k$, or $S_{k+1} \in E_{k,1,n_k}$ for some $n_k$ and there is no $T \subseteq S_{k+1}$ in $\bigcup_x E_{k,0,x}$. The existence of $S_{k+1}$ is guaranteed by Lemma \ref{lem:coh-rand-density}. Obviously, (1-3) above hold for $k+1$.

Clearly, every $X \in \bigcap_k [S_k]$ is as desired.
\end{proof}

\begin{remark}
By a relativization of Theorem 7 in Ku\v{c}era \cite{Kucera:85}, there exists $2$-random $X$ such that $\emptyset'' \leq_T X \oplus \emptyset' \leq_T X'$. Hence, there exists $2$-random $X$ which computes cohesive sets for all computable sequences, by Jockusch and Stephan \cite{Jockusch.Stephan:1993.cohesive}. So, it is natural to ask whether there is a $3$-random $X$ computing a cohesive set for the sequence in Theorem \ref{thm:coh-random}.
\end{remark}

\bibliographystyle{plain}

\end{document}